\documentclass[reqno]{amsart}
\usepackage{amsmath}
\usepackage{amscd,amsthm,amsfonts,amsopn,amssymb}
\usepackage{hyperref}
\usepackage{epsfig}
\usepackage{graphicx,verbatim}
\numberwithin{equation}{section}

\makeatletter
\def\subsection{\@startsection{subsection}{2}%
  \z@{0.0\linespacing}{-.5em}%
  {\normalfont\bfseries}}
\makeatother

\newtheorem{theorem}{Theorem}[section]

\newtheorem{proposition}[theorem]{Proposition}
\newtheorem{corollary}[theorem]{Corollary}
\newtheorem{definition}[theorem]{Definition}

\newtheorem{lemma}[theorem]{Lemma}

\theoremstyle{remark}
\newtheorem{remark}[theorem]{Remark}
\DeclareMathOperator{\supp}{supp\,}

\long\def\symbolfootnote[#1]#2{\begingroup%
\def\thefootnote{\fnsymbol{footnote}}\footnote[#1]{#2}\endgroup} 

\allowdisplaybreaks
\begin{document}
\author{Aynur Bulut}\address{The University of Texas at Austin, 1 University Station C1200, Austin, TX 78712}\email{abulut@math.utexas.edu}
\title[The radial defocusing energy-supercritical cubic NLW in $\mathbb{R}^{1+5}$]{The radial defocusing energy-supercritical cubic nonlinear wave equation in $\mathbb{R}^{1+5}$}
\thanks{\today}
\begin{abstract}
In this work, we consider the energy-supercritical defocusing cubic nonlinear wave equation in dimension $d=5$ for radially symmetric initial data.  We prove that an a priori bound in the critical space implies global well-posedness and scattering.  The main tool that we use is a frequency localized version of the classical Morawetz inequality, inspired by recent developments in the study of the mass and energy critical nonlinear Schr\"odinger equation.
\end{abstract}

\maketitle

\section{Introduction}
In this paper, we continue our study of the energy-supercritical defocusing cubic nonlinear wave equation, initiated in \cite{BulutCubic}.  In that work, our goal was to prove that an a priori bound in the critical space leads to global well-posedness and scattering for solutions to the initial value problem
\begin{align*}
(IVP)\quad\left\lbrace\begin{array}{rl}u_{tt}-\Delta u+|u|^2u&=0,\\
(u(0),u_t(0))&=(u_0,u_1)\in \dot{H}_x^{s_c}(\mathbb{R}^d)\times\dot{H}_x^{s_c-1}(\mathbb{R}^d),
\end{array}\right.
\end{align*}
in dimensions $d\geq 6$, where $s_c=\frac{d-2}{2}$ and $u:I\times\mathbb{R}^d\rightarrow\mathbb{R}$ with $0\in I\subset \mathbb{R}$.  There is a natural scaling associated to (IVP).  That is to say, if we define
\begin{align*}
u_\lambda(t,x):=\lambda u(\lambda t,\lambda x)
\end{align*}
then the map $u\mapsto u_\lambda$ carries the set of solutions of (IVP) to itself.  Moreover, this map preserves the $\dot{H}_x^{s_c}\times\dot{H}_x^{s_c-1}$ norm of the initial data, and therefore the space $\dot{H}_x^{s_c}\times\dot{H}_x^{s_c-1}$ is referred to as the {\it critical space} with respect to the scaling.  We also recall that solutions to (IVP) conserve the {\it energy},
\begin{align*}
E(u(t),u_t(t))=\int_{\mathbb{R}^d} \frac{1}{2}|\nabla u(t)|^2+\frac{1}{2}|u_t(t)|^2+\frac{1}{4}|u(t)|^4dx,
\end{align*}
which is finite for solutions to (IVP) when $s_c=1$.  In view of this, we call (IVP) {\it energy-supercritical} when $s_c>1$, that is $d\geq 5$.

In \cite{BulutCubic}, our study to prove global well-posedness for (IVP) proceeded by making use of the concentration compactness approach introduced by Kenig and Merle \cite{KenigMerleNLS,KenigMerleNLW}, reducing the question to an analysis of three specific blowup scenarios as in \cite{KV3,KVradial}.  The key part of this analysis was to show that in each of these scenarios solutions have finite energy, for which a major tool was the double Duhamel technique \cite{CKSTT,KillipVisanECritical,KillipVisanSupercriticalNLS}.  This was the source of the restriction in our considerations to dimensions $d\geq 6$.

In the present work, we extend this result to include dimension $d=5$ in the case of radial initial data.  More precisely, we consider
\begin{align*}
(NLW)\quad\left\lbrace\begin{array}{rl}u_{tt}-\Delta u+|u|^2u&=0,\\
(u(0),u_t(0))&=(u_0,u_1)\in \dot{H}_x^{3/2}(\mathbb{R}^5)\times\dot{H}_x^{1/2}(\mathbb{R}^5),
\end{array}\right.
\end{align*}
where $u:I\times\mathbb{R}^5\rightarrow\mathbb{R}$, $0\in I\subset \mathbb{R}$, and $(u_0,u_1)$ is radially symmetric.  

We will use the following notion of solution to (NLW):
\begin{definition}
\label{def_solution}
We say that $u:I\times\mathbb{R}^5\rightarrow\mathbb{R}$ with $0\in I\subset\mathbb{R}$ is a {\it solution} to (NLW) if $(u,u_t)$ belongs to $C_t(K;\dot{H}_x^{3/2}\times\dot{H}_x^{1/2})\cap L_{t,x}^{6}(K\times\mathbb{R}^5)$ for every $K\subset I$ compact and $u$ satisfies the {\it Duhamel formula}
\begin{align*}
u(t)&=\cos(t|\nabla|)u_0+\frac{\sin(t|\nabla|)}{|\nabla|}u_1+\int_0^t \frac{\sin((t-t')|\nabla|)}{|\nabla|}F(u(t'))dt'
\end{align*}
for every $t\in I$.
\end{definition}
We now state the main result of this paper:
\begin{theorem}
\label{thm_main}
Suppose $u:I\times\mathbb{R}^5\rightarrow\mathbb{R}$ is a solution to (NLW) with radial initial data, maximal interval of existence $I\subset \mathbb{R}$, and satisfying the a priori bound
\begin{align*}
(u,u_t)\in L_t^\infty(I;\dot{H}_x^{3/2}\times\dot{H}_x^{1/2}).
\end{align*}
Then $u$ is global and 
\begin{align*}
\lVert u\rVert_{L_{t,x}^6(\mathbb{R}\times\mathbb{R}^5)}\leq C
\end{align*}
for some constant $C=C(\lVert (u,u_t)\rVert_{L_t^\infty (\dot{H}_x^{3/2}\times\dot{H}_x^{1/2})})$.  Furthermore, $u$ scatters both forward and backward in time.
\end{theorem}
As in the high dimensional case treated in \cite{BulutCubic}, our proof of Thoerem $\ref{thm_main}$ is a proof by contradiction following the concentration compactness approach of Kenig and Merle.  The key ingredient which allows us to work in dimension five is the use of a frequency localized version of the Morawetz estimate, inspired by recent progress in the study of the global well-posedness problem for the mass and energy critical nonlinear Schr\"odinger equation \cite{Dodson,KV3nls,Visan4}.  Equipped with this inequality and the assumption of radial symmetry, we bypass the need to prove the finiteness of energy.  We plan to address the case of general initial data in dimension five in a future work.

We note that the problem (IVP) in the radial energy-supercritical setting has recently been treated by Killip and Visan in \cite{KVradial}, for a range of nonlinearities $|u|^pu$ dependent on the dimension.  In that work, the restriction on $p$ corresponding to five spatial dimensions is $\frac{4}{3}<p<2$, excluding the cubic case treated in the present work.

We also recently learned that Kenig and Merle have treated the defocusing energy-supercritical NLW with the quintic nonlinearity and radial data in all odd dimensions \cite{KenigMerleOddDimensions}.

We now recall the definition of the class of almost periodic solutions.
\begin{definition}
A solution $u$ to (NLW) with time interval $I$ is said to be almost periodic modulo symmetries if $(u,u_t)\in L_t^\infty(I;\dot{H}_x^{3/2}\times\dot{H}_x^{1/2})$ and there exist functions $N:I\rightarrow\mathbb{R}^+$, $x:I\rightarrow\mathbb{R}^5$ and $C:\mathbb{R}^+\rightarrow\mathbb{R}^+$ such that for all $t\in I$ and $\eta>0$,
\begin{align*}
\int_{|x-x(t)|\geq C(\eta)/N(t)} ||\nabla|^{3/2}u(t,x)|^2+||\nabla|^{1/2}u_t(t,x)|^2dx&\leq \eta,
\end{align*}
and
\begin{align*}
\int_{|\xi|\geq C(\eta)N(t)} |\xi|^{3}|\hat{u}(t,\xi)|^2+|\xi|\,|\hat{u}_t(t,\xi)|^2d\xi&\leq \eta.
\end{align*}
\end{definition}
We next remark a consequence of the notion of almost periodicity.
\begin{remark}
\label{rem_apalt}
The property of almost periodicity is equivalent to the following condition: there exist functions $N:I\rightarrow \mathbb{R}^+$ and $x:I\rightarrow\mathbb{R}^5$ such that the set
\begin{align}
K&=\{(\frac{1}{N(t)}u(t,x(t)+\frac{x}{N(t)}),\,\frac{1}{N(t)^{2}}u_t(t,x(t)+\frac{x}{N(t)})):t\in I\},\label{cptness}
\end{align}
has compact closure in $\dot{H}_x^{3/2}(\mathbb{R}^5)\times\dot{H}_x^{1/2}(\mathbb{R}^5)$.  In particular, if $u$ is almost periodic, then for every $\eta>0$ there exists $C(\eta)>0$ such that for all $t\in I$,
\begin{align*}
\int_{|x-x(t)|\geq C(\eta)/N(t)} |\nabla u(t,x)|^{5/2}dx+\int_{|x-x(t)|\geq C(\eta)/N(t)} |u_t(t,x)|^{5/2}dx&\leq \eta.
\end{align*}
\end{remark}

With this definition in hand, we are now ready to outline our strategy for proving Theorem $\ref{thm_main}$.  We first recall the following result due to Kenig and Merle \cite{KM}, which shows that the failure of Theorem $\ref{thm_main}$ gives the existence of a minimal counterexample which belongs to the class of almost periodic solutions.
\begin{theorem}\cite{KM}
\label{thm_minimal}
Suppose that Theorem $\ref{thm_main}$ failed.  Then there exists a radial solution $u:I\times\mathbb{R}^5\rightarrow\mathbb{R}$ to (NLW) with maximal interval of existence $I$,
\begin{align*}
(u,u_t)\in L_t^\infty(I;\dot{H}_x^{3/2}\times\dot{H}_x^{1/2}),\quad\textrm{and}\quad \lVert u\rVert_{L_{t,x}^{6}(I\times\mathbb{R}^5)}=\infty
\end{align*}
such that $u$ is a minimal blow-up solution in the following sense: for any solution $v$ with maximal interval of existence $J$ such that $\lVert v\rVert_{L_{t,x}^{6}(J\times\mathbb{R}^d)}=\infty$, we have
\begin{align*}
\sup_{t\in I} \lVert (u(t),u_t(t))\rVert_{\dot{H}_x^{3/2}\times\dot{H}_x^{1/2}}&\leq \sup_{t\in J} \lVert (v(t),v_t(t))\rVert_{\dot{H}_x^{3/2}\times\dot{H}_x^{1/2}}.
\end{align*}
Moreover, $u$ is almost periodic modulo symmetries.
\end{theorem}

We will also use the following refinement of Theorem $\ref{thm_minimal}$, which shows that the almost periodic solution $u$ and associated function $N(t)$ can be chosen so that $N(t)$ is piecewise constant on $I^+:=I\cap [0,\infty)$ and $N(t)\geq 1$ for all $t$ in this set.
\begin{theorem}
\label{thm_char_intervals}
Suppose that Theorem $\ref{thm_main}$ failed.  Then there exists a radial solution $u:I\times\mathbb{R}^5\rightarrow\mathbb{R}$ to (NLW) with maximal interval of existence $I$ such that $u$ is almost periodic modulo symmetries, $(u,u_t)\in L_t^\infty(I;\dot{H}_x^{3/2}\times\dot{H}_x^{1/2})$, $\lVert u\rVert_{L_{t,x}^6(I\times\mathbb{R}^5)}=\infty$, and there exists $\delta>0$ and a family of disjoint intervals $\{J_k\}_{k\geq 1}$ with $I^+=\cup J_k$,
\begin{align*}
N(t)=N_k\geq 1\,\,\textrm{for}\,\,t\in J_k,\quad\textrm{and}\quad |J_k|=\delta N_k^{-1}.
\end{align*}
Moreover, either 
\begin{align*}
|I^+|<\infty\quad\textrm{or}\quad |I^+|=\infty.
\end{align*}
\end{theorem}
This theorem is proved by applying a rescaling argument to the function obtained in Theorem $\ref{thm_minimal}$ to find another almost periodic solution with $N(t)\geq 1$ for $t\in I^+$ (see Theorem $7.1$ in \cite{KenigMerleNLW}).  One then observes that the function $N(t)$ obeys $N(s)\sim_u N(t)$ for $|s-t|\leq \delta N(t)^{-1}$ and $\delta$ suitably chosen,  as a consequence of the scaling symmetry and local theory for (NLW).  This property is proved in the NLS setting in \cite[Corollary $3.6$]{KillipTaoVisan}; however, the arguments apply equally to (NLW).  After a suitable modification of $N(t)$ and $C(t)$, the desired result is obtained.

In Theorem $\ref{thm_char_intervals}$ we divide the solutions of (NLW) into two classes depending on the control granted by the frequency localized Morawetz estimate, Lemma $\ref{lem1}$.  This is inspired by recent works in the mass and energy critical NLS settings \cite{Dodson,Visan4}.  In the present context, this corresponds to distinguishing the cases $|I^+|<\infty$ and $|I^+|=\infty$; we also note that this distinction is also present in \cite{KenigMerleNLW}.

We next give a quick remark concerning the decay of norms of the Littlewood-Paley projections of $u$.
\begin{remark}
\label{rem_limit}  Suppose that $u$ is as in Theorem $\ref{thm_char_intervals}$.  The property $\inf_{t\in I^+} N(t)=\inf_k N_k\geq 1$ along with the definition of almost periodicity implies
\begin{align*}
\lim_{N\rightarrow 0} [\lVert u_{\leq N}\rVert_{L_t^\infty(I^+;\dot{H}_x^{3/2})}+\lVert P_{\leq N}u_t\rVert_{L_t^\infty (I^+;\dot{H}_x^{1/2})}]=0.
\end{align*}
\end{remark}

The proof of Theorem $\ref{thm_main}$ is therefore reduced to the task of showing that solutions satisfying the properties given in Theorem $\ref{thm_char_intervals}$ cannot occur.  This is accomplished in Sections $\ref{sec_fintime}$ and $\ref{sec_inftime}$ below, corresponding to the cases $|I^+|<\infty$ and $|I^+|=\infty$, respectively.

To handle the case $|I^+|<\infty$, we show that the solution at time $t$ must be supported in space inside a ball centered at the origin with radius shrinking to $0$ as $t$ approaches the blowup time.  This is then shown to be incompatible with the conservation of energy.

On the other hand, the case $|I^+|=\infty$ requires significantly more analysis.  For this case, we observe that, given $\eta>0$, the frequency localized Morawetz estimate obtained in Section $\ref{sec_mor}$ implies the bound
\begin{align*}
\int_{I_0}\int_{\mathbb{R}^5} \frac{|u_{\geq N}(t)|^4}{|x|}dxdt&\lesssim_u \eta(N^{-1}+|I_0|)
\end{align*}
for $N$ sufficiently small and all $I_0\subset I^+$ compact.  We then obtain a bound from below on the left hand side of this inequality by a multiple of $|I_0|$ (up to a small error term).  Choosing $\eta$ sufficiently small then gives the desired contradiction.

We now conclude this section by giving an outline of the remainder of the paper.  In Section $\ref{sec_prelim}$, we recall some preliminaries and establish our notation.  Section $\ref{sec_str}$ is then devoted to the proof of a frequency localized version of the Strichartz inequality which will be essential to obtain the frequency localized Morawetz estimate.  This estimate is then proved in Section $\ref{sec_mor}$.  Sections $\ref{sec_fintime}$ and $\ref{sec_inftime}$ then preclude the existence of the finite time and infinite time blowup solutions identified in Theorem $\ref{thm_char_intervals}$, completing the proof of Theorem $\ref{thm_main}$.

\subsection{Acknowledgements}

The author would like to thank M. Visan as well as N. Pavlovi\'c and W. Beckner for many enlightening and valuable discussions concerning the content of this paper.  The author also wishes to thank M. Visan for her careful reading and comments on the manuscript.  The author was supported by a Graduate School Continuing Fellowship from the University of Texas at Austin during the preparation of this work.

\section{Preliminaries}
\label{sec_prelim}
In this section, we introduce the notation and some standard estimates that we use throughout the paper.  We write $L_t^qL_x^r$ to indicate the space-time norm
\begin{align*}
\lVert u\rVert_{L_t^qL_x^r}=\left(\int_{\mathbb{R}} \lVert u(t)\rVert_{L_x^q}^rdt\right)^{1/r}
\end{align*}
with the standard convention when $q$ or $r$ is equal to infinity.  If $q=r$, we shorten the notation and write $L_{t,x}^q$.

We write $X\lesssim Y$ to mean that there exists a constant $C>0$ such that $X\leq CY$, while $X\lesssim_u Y$ indicates that the constant $C=C(u)$ may depend on $u$.  We use the symbol $\nabla$ for the derivative operator in only the space variables.  

Throughout the exposition, we define the Fourier transform on $\mathbb{R}^5$ by
\begin{align*}
\widehat{f}(\xi)=(2\pi)^{-5/2}\int_{\mathbb{R}^5} e^{-ix\cdot \xi}f(x)dx.
\end{align*}
We also denote the homogeneous Sobolev spaces by $\dot{H}_x^s(\mathbb{R}^5)$, $s\in\mathbb{R}$, equipped with the norm
\begin{align*}
\lVert f\rVert_{\dot{H}_x^s}=\lVert |\nabla|^sf\rVert_{L_x^2}
\end{align*}
where the fractional differentiation operator is given by
\begin{align*}
\widehat{|\nabla|^sf}(\xi)&=|\xi|^s\widehat{f}(\xi).
\end{align*}

For $s\geq 0$, we say that a pair of exponents $(q,r)$ is $\dot{H}_x^s$-{\it wave admissible} if $q,r\geq 2$, $r<\infty$ and it satisfies
\begin{align*}
\frac{1}{q}+\frac{2}{r}&\leq 1,\\
\frac{1}{q}+\frac{5}{r}&=\frac{5}{2}-s.
\end{align*}
We also define the following {\it Strichartz norms}.  For each $I\subset\mathbb{R}$ and $s\geq 0$, we set
\begin{align*}
\lVert u\rVert_{S_s(I)}&=\sup_{(q,r)\,\dot{H}_x^s-\textrm{wave admissible}} \lVert u\rVert_{L_t^qL_x^r(I\times\mathbb{R}^5)},\\
\lVert u\rVert_{N_s(I)}&=\inf_{(q,r)\,\dot{H}_x^s-\textrm{wave admissible}} \lVert u\rVert_{L_t^{q'}L_x^{r'}(I\times\mathbb{R}^5)}.
\end{align*}

Suppose $u:I\times\mathbb{R}^5\rightarrow\mathbb{R}$ with time interval $0\in I\subset\mathbb{R}$ is a solution to the nonlinear wave equation
\begin{align*}
\left\lbrace \begin{array}{rl}u_{tt}-\Delta u+F&=0\\
(u,u_t)|_{t=0}&=(u_0,u_1)\in \dot{H}_x^{\mu}\times \dot{H}_x^{\mu-1}(\mathbb{R}^d), \quad \mu\in\mathbb{R}.
\end{array}
\right.
\end{align*}
Then for all $s,\tilde{s}\in\mathbb{R}$ we have the inhomogeneous Strichartz estimates \cite{GinibreVelo,KeelTao},
\begin{align}
\nonumber &\lVert |\nabla|^s u\rVert_{S_{\mu-s}(I)}+\lVert |\nabla|^{s-1}u_t\rVert_{S_{\mu-s}(I)}\\
&\hspace{1.5in}\lesssim \lVert (u_0,u_1)\rVert_{\dot{H}_x^{\mu}\times\dot{H}_x^{\mu-1}}+\lVert |\nabla|^{\tilde{s}}F\rVert_{N_{1+\tilde{s}-\mu}(I)}.\label{str}
\end{align}

We also note that our assumption $u\in L_{t,x}^6(K\times\mathbb{R}^5)$ on the solution in Definition $\ref{def_solution}$, combined with the local theory and the Strichartz estimates, implies
\begin{align*}
\lVert |\nabla|^s u\rVert_{S_{\frac{3}{2}-s}(K)}<\infty
\end{align*}
for $s\in [0,\frac{3}{2}]$ and $K\subset I$ compact.

Moreover, for every almost periodic solution $u$ to (NLW) there exists $C(u)>0$ such that for every compact $K\subset I$
\begin{align*}
\frac{1}{C(u)}\int_K N(t)dt\leq \lVert u\rVert_{L_t^6(K;L_x^6)}^6\leq C(u)\bigg(1+\int_K N(t)dt\bigg),
\end{align*}
together with the bound
\begin{align}
\frac{1}{C(u)}\int_{K} N(t)dt\leq \lVert |\nabla|^{3/4}u\rVert_{L_t^2(K;L_x^4)}^2\leq C_1(u)\bigg(1+\int_K N(t)dt\bigg).\label{eqI}
\end{align}
The above bounds are consequences of almost periodicity and the Strichartz estimates ($\ref{str}$).  In the NLS setting, we refer to the analogous estimates in \cite[Lemma $5.21$]{KillipVisanClayNotes} and \cite[Lemma $1.7$]{Visan4}, while for solutions to (NLW) these bounds are obtained in a similar manner, after accounting for the difference in scaling between the equations.

We next recall some basic facts from the Littlewood-Paley theory that will be used frequently in the sequel.   Let $\phi(\xi)$ be a real valued radially symmetric bump function supported in the ball $\{\xi\in\mathbb{R}^d:|\xi|\leq \frac{11}{10}\}$ which equals $1$ on the ball $\{\xi\in\mathbb{R}^d:|\xi|\leq 1\}$.  For any dyadic number $N=2^k$, $k\in\mathbb{Z}$, we define the following Littlewood-Paley operators:
\begin{align*}
\widehat{P_{\leq N}f}(\xi)&=\phi(\xi/N)\hat{f}(\xi),\\
\widehat{P_{>N}f}(\xi)&=(1-\phi(\xi/N)\hat{f}(\xi),\\
\widehat{P_Nf}(\xi)&=(\phi(\xi/N)-\phi(2\xi/N))\hat{f}(\xi).
\end{align*}

\noindent Similarly, we define $P_{<N}$ and $P_{\geq N}$ with 
\begin{align*}
P_{<N}=P_{\leq N}-P_N,\quad P_{\geq N}=P_{>N}+P_N,
\end{align*}
and also
\begin{align*}
P_{M<\cdot\leq N}:=P_{\leq N}-P_{\leq M}=\sum_{M<N_1\leq N} P_{N_1}
\end{align*}
whenever $M\leq N$.

These operators commute with one another and with derivative operators.  Moreover, they are bounded on $L^p$ for $1\leq p\leq \infty$ and obey the following {\it Bernstein inequalities}, 
\begin{align*}
\lVert |\nabla|^s P_{\leq N}f\rVert_{L_x^p}&\lesssim N^s\lVert P_{\leq N}f\rVert_{L_x^p},\\
\lVert P_{>N}f\rVert_{L_x^p}&\lesssim N^{-s}\lVert P_{>N}|\nabla|^sf\rVert_{L_x^p},\\
\lVert |\nabla|^{\pm s}P_Nf\rVert_{L_x^p}&\sim N^{\pm s} \lVert P_Nf\rVert_{L_x^p},
\end{align*}
with $s\geq 0$ and $1\leq p\leq\infty$.

\section{Frequency localized Strichartz estimate}\label{sec_str}

We now obtain a frequency localized version of the Strichartz estimates that we will use as a main ingredient in proving the frequency localized Morawetz estimate in Section $\ref{sec_mor}$.  The proof of this result is inspired by analogous results for the mass and energy critical nonlinear Schr\"odinger equation due to Dodson \cite{Dodson} and Visan \cite{Visan4}.

\begin{theorem}[Frequency localized Strichartz estimate.]
\label{floc_str}
Suppose that $u$ is an almost periodic solution to (NLW) with maximal interval of existence $I$, $(u,u_t)\in L_t^\infty(I;\dot{H}_x^{3/2}\times\dot{H}_x^{1/2})$, and such that there exist disjoint intervals $\{J_k\}_{k\geq 1}$ with $I^+=\cup J_k$ and for every $k$, $N(t)=N_k\in [1,\infty)$ on $J_k$, $|J_k|=\delta N_k^{-1}$.

Then there exists $C=C(u)>0$ such that for all dyadic $N$ and compact intervals $I_0=\cup J_k\subset I^+$ we have
\begin{align}
\lVert |\nabla|^{3/4} u_{\leq N}\rVert_{L_t^2(I_0;L_x^4)}\leq C(u)(1+(N|I_0|)^{1/2})\label{goal1}
\end{align}
Moreover, for every $\eta>0$ there exists $N_0>0$ such that for $N<N_0$ we have
\begin{align}
\lVert |\nabla|^{3/4}u_{\leq N}\rVert_{L_t^2(I_0;L_x^4)}\leq C(u)\eta (1+(N|I_0|)^{1/2}).\label{goal2}
\end{align}
\end{theorem}

Before we proceed with the proof of the theorem, we record the following related estimates, derived by interpolating $(\ref{goal1})$ and $(\ref{goal2})$ with the a priori bound on $L_t^\infty (\dot{H}_x^{3/2}\times\dot{H}_x^{1/2})$.

\begin{corollary}
\label{cor_floc_str}
Let $u$ be as in Theorem $\ref{floc_str}$.  Then there exists $C(u)>0$ such that
\begin{itemize}
\item for each dyadic $N>0$ and compact interval $I_0=\cup J_k\subset I^+$ we have
\begin{align*}
\lVert u_{>N}\rVert_{L_t^3(I_0;L_x^{30/7})}&\leq C(u)N^{-1/2}(1+N|I_0|)^{1/3},\\
\lVert u_{>N}\rVert_{L_t^4(I_0;L_x^{20/7})}&\leq C(u)N^{-1}(1+N|I_0|)^{1/4},
\end{align*}
and
\item for each $\eta>0$ there exists $N_0>0$ such that for $N<N_0$ we have
\begin{align*}
\lVert \nabla u_{\leq N}\rVert_{L_t^3(I_0;L_x^3)}&\leq C(u)\eta(1+N|I_0|)^{1/3},\\
\lVert \nabla u_{\leq N}\rVert_{L_t^4(I_0;L_x^{20/7})}&\leq C(u)\eta (1+N|I_0|)^{1/4}.
\end{align*}
\end{itemize}
\end{corollary}

\begin{proof}[Proof of Theorem $\ref{floc_str}$.]  We begin by showing ($\ref{goal1}$).  Let $I_0\subset I^+$ be given as stated and observe that the bound ($\ref{eqI}$) implies that ($\ref{goal1}$) holds with $C_1(u)$ for all
\begin{align*}
N\geq \frac{\int_{I_0}N(t)dt}{|I_0|}.
\end{align*}

For general dyadic numbers $N$, we proceed by induction.  Fix 
\begin{align*}
C(u)>\max \{C_1(u),1\}
\end{align*}
to be determined, and suppose that ($\ref{goal1}$) holds for all $N$ larger than some $N_0$.  Our goal is to show that ($\ref{goal1}$) holds for $N=N_1:=N_0/2$ (with $C(u)$ unchanged).  Toward this end, we apply the Strichartz inequality to obtain
\begin{align}
\nonumber &\lVert |\nabla|^{3/4} u_{\leq N_1}\rVert_{L_t^2(I_0;L_x^4)}\\
&\hspace{0.2in}\lesssim \inf_{t\in I_0} \lVert (u_{\leq N_1}(t),\partial_t u_{\leq N_1}(t))\rVert_{\dot{H}_x^{3/2}\times\dot{H}_x^{1/2}}+\lVert |\nabla|^{5/4} P_{\leq N_1}[u(t)^3]\rVert_{L_t^2(I_0; L_x^\frac{4}{3})}.\label{eqA01}
\end{align}
In the remainder of the proof, all space-time norms will be over the set $I_0\times\mathbb{R}^5$, unless otherwise indicated.

To estimate the nonlinear term in ($\ref{eqA01}$), we fix $0<\eta_0\leq \frac{1}{2}$ (to be determined later in the argument) and use the almost periodicity of $u$ to choose $c_0=c_0(\eta_0)$ such that 
\begin{align}
\lVert |\nabla|^{3/2} u_{\leq c_0N(t)}(t)\rVert_{L_t^\infty L_x^2}+\lVert |\nabla|^{1/2}\partial_t u_{\leq c_0N(t)}(t)\rVert_{L_t^\infty L_x^2}\leq \eta_0.\label{eqA0}
\end{align}
Then, writing
\begin{align*}
u(t)&=u_{\leq N_1/\eta_0}(t)+u_{>N_1/\eta_0}(t)
\end{align*}
and using the identity
\begin{align*}
&(u_{>N_1/\eta_0}(t)+u_{\leq N_1/\eta_0}(t))^3\\
&\hspace{0.2in}=u_{>N_1/\eta_0}(t)^3+3u_{>N_1/\eta_0}(t)u_{\leq N_1/\eta_0}(t)u(t)+u_{\leq N_1/\eta_0}(t)^3,
\end{align*}
we obtain
\begin{align}
\nonumber \lVert |\nabla|^{3/4}u_{\leq N_1}\rVert_{L_t^2L_x^4}&\lesssim \inf_{t\in I_0} \lVert (u_{\leq N_1}(t),\partial_t u_{\leq N_1}(t))\rVert_{\dot{H}_x^{3/2}\times\dot{H}_x^{1/2}}\\
\label{eqA}&\hspace{0.2in}+\lVert |\nabla|^{5/4}P_{\leq N_1}[u_{>N_1/\eta_0}(t)^3]\rVert_{L_t^2 L_x^\frac{4}{3}}\\
\label{eqB}&\hspace{0.2in}+\lVert |\nabla|^{5/4}P_{\leq N_1}[u_{>N_1/\eta_0}(t)u_{\leq N_1/\eta_0}(t)u(t)]\rVert_{L_t^2 L_x^\frac{4}{3}}\\
\nonumber &\hspace{0.2in}+\lVert |\nabla|^{5/4}P_{\leq N_1}[u_{\leq N_1/\eta_0}(t)^3]\rVert_{L_t^2 L_x^\frac{4}{3}}.
\end{align}
Furthermore, we bound the last term by a multiple of 
\begin{align}
\label{eqC}&\lVert |\nabla|^{5/4}P_{\leq N_1}\bigg(u_{\leq N_1/\eta_0}(t)\big[P_{\leq}u_{<c_0N(t)}(t)\big]^2\bigg)\rVert_{L_t^2 L_x^\frac{4}{3}}\\
\label{eqD}&\hspace{0.2in}+\lVert |\nabla|^{5/4}P_{\leq N_1}\bigg(u_{\leq N_1/\eta_0}(t)\big[P_{\leq}u_{\leq c_0N(t)}(t)\big]\big[P_{\leq}u_{\geq c_0N(t)}(t)\big]\bigg)\rVert_{L_t^2 L_x^\frac{4}{3}}\\
\label{eqE}&\hspace{0.2in}+\lVert |\nabla|^{5/4}P_{\leq N_1}\bigg(u_{\leq N_1/\eta_0}(t)\big[P_{\leq}u_{\geq c_0N(t)}(t)\big]^2\bigg)\rVert_{L_t^2 L_x^\frac{4}{3}},
\end{align}
where we have set $P_{\leq}=P_{\leq N_1/\eta_0}$ and used the decomposition
\begin{align*}
P_{\leq}u(t)&=P_{\leq }u_{\leq c_0N(t)}(t)+P_{\leq }u_{> c_0N(t)}(t)
\end{align*}
and where $c_0$ is chosen in $(\ref{eqA0})$.

Thus, it suffices to bound $(\ref{eqA})$ through $(\ref{eqE})$.  Before estimating each of these terms, we will need the following estimate, which is obtained via H\"older's inequality in time and interpolation: for each dyadic $M>0$,
\begin{align}
\nonumber \lVert |\nabla|^{5/4}u_{\leq M}\rVert_{L_t^2L_x^{20/7}}&\leq (M|I_0|)^{1/4}\lVert \nabla u_{\leq M}\rVert_{L_t^4L_x^{20/7}}\\
\nonumber &\lesssim (M|I_0|)^{1/2}+\lVert \nabla u_{\leq M}\rVert_{L_t^4L_x^{20/7}}^2\\
\nonumber &\lesssim (M|I_0|)^{1/2}+\lVert |\nabla|^{3/4}u_{\leq M}\rVert_{L_t^{2}L_x^{4}}\lVert |\nabla|^{5/4}u_{\leq M}\rVert_{L_t^\infty L_x^{20/9}}\\
\label{eqF}&\lesssim_u (M|I_0|)^{1/2}+\lVert |\nabla|^{3/4}u_{\leq M}\rVert_{L_t^{2}L_x^{4}}.
\end{align}

With this bound in hand, we are now ready to estimate the above terms.  For $(\ref{eqA})$, we note that an application of Bernstein's inequality gives
\begin{align}
\nonumber (\ref{eqA})&\lesssim N_1^{5/4}\lVert u_{>N_1/\eta_0}(t)^3\rVert_{L_t^2L_x^{4/3}}\\
\nonumber &\lesssim N_1^{5/4} \sum_{M>N_1/\eta_0}\lVert u_{M}(t)\rVert_{L_t^2L_x^{20/7}}\lVert u_{>N_1/\eta_0}(t)\rVert_{L_t^\infty L_x^5}^2\\
\nonumber &\lesssim_u N_1^{5/4} \sum_{M>N_1/\eta_0} M^{-5/4}\lVert |\nabla|^{5/4}u_M\rVert_{L_t^2L_x^{20/7}}\\
\nonumber &\leq\eta_0^{3/4}C_2(u)C(u)(N_1|I_0|)^\frac{1}{2}+\eta_0^{5/4}C_2(u)C(u),
\end{align}
where to obtain the last line we have used $(\ref{eqF})$ and the induction hypothesis.  We may use the same argument to estimate $(\ref{eqB})$, obtaining
\begin{align*}
(\ref{eqB})&\leq \eta_0^{3/4}C_2(u)C(u)(N_1|I_0|)^{1/2}+\eta_0^{5/4}C_2(u)C(u).
\end{align*}
On the other hand, to estimate $(\ref{eqC})$, we apply the fractional product rule \cite{ChristWeinstein,KenigPonceVega} and the Sobolev embedding to obtain
\begin{align}
\nonumber (\ref{eqC})&\leq \lVert |\nabla|^{5/4} u_{\leq N_1/\eta_0}\rVert_{L_t^2L_x^{20/7}}\lVert [P_{\leq}u_{\leq c_0N(t)}(t)]^2\rVert_{L_t^\infty L_x^{5/2}}\\
\nonumber &\hspace{0.2in}+\lVert u_{\leq N_1/\eta_0}\rVert_{L_t^2L_x^{10}}\lVert |\nabla|^{5/4}\left([P_{\leq }u_{\leq c_0N(t)}(t)]^2\right)\rVert_{L_t^{\infty}L_x^{20/13}}\\
\nonumber &\lesssim_u \eta_0^2\left(\lVert |\nabla|^{5/4}u_{\leq N_1/\eta_0}\rVert_{L_t^2L_x^{20/7}}+\lVert |\nabla|^{3/4}u_{\leq N_1/\eta_0}\rVert_{L_t^2L_x^4}\right),
\end{align}
where $c_0$ is chosen in $(\ref{eqA0})$.  Then, using ($\ref{eqF}$) and the induction hypothesis once again, we get
\begin{align}
\nonumber (\ref{eqC})&\leq \eta_0^2C_3(u)(\eta_0^{-1/2}C(u)(N_1|I_0|)^{1/2}+C(u)).
\end{align}

We now turn our attention to the two remaining terms.  In what follows, we will use the notation $v(t)$ to refer to either of the function $P_{\leq N_1/\eta_0}u_{\leq c_0N(t)}(t)$ and $P_{\leq N_1/\eta_0}u_{>c_0N(t)}(t)$.  In particular, using Bernstein's inequalities combined with the fractional product rule, we obtain the preliminary bound
\begin{align}
\nonumber &\lVert |\nabla|^{5/4}P_{\leq N_1}\big[ u_{\leq N_1/\eta_0}(t)v(t)P_{\leq }u_{> c_0N(t)}(t)\big]\rVert_{L_t^2(J_k;L_x^\frac{4}{3})}\\
\nonumber &\hspace{0.1in}\lesssim N_1^{1/2}\lVert |\nabla|^{3/4}\big[u_{\leq N_1/\eta_0}(t)v(t)\big]\rVert_{L_t^\infty L_x^{20/11}}\lVert P_{\leq}u_{> c_0N(t)}(t)\rVert_{L_t^2(J_k;L_x^5)}\\
\label{eqG}&\hspace{0.2in}+N_1^{1/2}\lVert u_{\leq N_1/\eta_0}\rVert_{L_t^6(J_k;L_x^6)}\lVert u\rVert_{L_t^6(J_k;L_x^6)}\lVert |\nabla|^{3/4}P_{\leq} u_{>c_0N(t)}\rVert_{L_t^6(J_k;L_x^{12/5})}.
\end{align}
We then use the fractional product rule again to bound the factor
\begin{align*}
\lVert |\nabla|^{3/4}[u_{\leq N_1/\eta_0}v]\rVert_{L_t^\infty L_x^{20/11}}&\lesssim \lVert |\nabla|^{3/4} u_{\leq N_1/\eta_0}\rVert_{L_t^\infty L_x^{20/7}}\lVert v\rVert_{L_t^\infty L_x^5}\\
&\hspace{0.2in}+\lVert u_{\leq N_1/\eta_0}\rVert_{L_t^\infty L_x^5}\lVert |\nabla|^{3/4}v\rVert_{L_t^\infty L_x^{20/7}}\\
&\lesssim_u 1.
\end{align*}
Invoking this bound in ($\ref{eqG}$), we obtain
\begin{align}
\nonumber &\max\{(\ref{eqD}),(\ref{eqE})\}\\
\nonumber &\leq \left(\sum_{J_k\subset I_0} \lVert |\nabla|^{5/4}P_{\leq N_1}\big[ u_{\leq N_1/\eta_0}(t)v(t)P_{\leq}u_{>c_0N(t)}(t)\big]\rVert_{L_t^2(J_k;L_x^{4/3})}^2\right)^{1/2}\\
\label{eqH}&\lesssim_u N^{1/2}\left(\sum_{J_K\subset I_0} \left\{\lVert u_{>c_0N_k}\rVert_{L_t^2(J_k;L_x^5)}^2+\lVert |\nabla|^{3/4}u_{>c_0N_k}\rVert_{L_t^6(J_k;L_x^{12/5})}^2\right\}\right)^{1/2}
\end{align}
where in the second term of $(\ref{eqG})$ we use the bound $(\ref{eqI})$ in the form
\begin{align}
\lVert u\rVert_{L_t^6(J_k;L_x^6)}\leq C(u)(1+\delta)\lesssim_u 1.\label{eqIk}
\end{align}

Moreover, using Bernstein's inequalities and the bounds $\lVert |\nabla|^{1/2}u\rVert_{L_t^2(J_k;L_x^5)}\lesssim_u 1$ and $\lVert |\nabla|^{5/4}u\rVert_{L_t^6(J_k;L_x^{12/5})}\lesssim_u 1$ (these bounds are obtained from the same argument used to prove ($\ref{eqI}$)),
\begin{align*}
(\ref{eqH})&\lesssim_u N^{1/2}\left(\sum_{J_k\subset I_0} \frac{1}{c_0N_k}\left\{\lVert |\nabla|^{1/2} u_{>c_0N_k}\rVert_{L_t^2(J_k;L_x^5)}^2\right.\right.\\
&\hspace{1in}\left.\left.+\lVert |\nabla|^{5/4}u_{>c_0N_k}\rVert_{L_t^6(J_k;L_x^{12/5})}^2\right\}\right)^{1/2}\\
&\lesssim_u N^{1/2}\left(\sum_{J_k\subset I_0} \frac{1}{c_0N_k}\right)^{1/2}\\
&\leq \frac{C_4(u)}{c_0^{1/2}}(N_1|I_0|)^{1/2}.
\end{align*}

Combining the estimates of ($\ref{eqA}$) through ($\ref{eqE}$), we obtain
\begin{align}
\nonumber &\lVert |\nabla|^{3/4} u_{\leq N_1}\rVert_{L_t^2L_x^4}\\
\nonumber &\hspace{0.2in}\leq C_0\inf_{t\in I_0} \lVert (u_{\leq N_1}(t),\partial_t u_{\leq N_1}(t))\rVert_{\dot{H}_x^{3/2}\times\dot{H}_x^{1/2}}\\
\nonumber &\hspace{0.4in}+2\eta_0^{3/4}C_2(u)C(u)((N_1|I_0|)^{1/2}+\eta_0^{1/2})\\
&\hspace{0.4in}+\eta_0^2C_3(u)C(u)(\eta_0^{-1/2}(N_1|I_0|)^{1/2}+1)+\frac{C_4(u)}{c_0^{1/2}}(N_1|I_0|)^{1/2}.\label{eqabcad1}
\end{align}
We now choose $\eta_0$ sufficiently small (depending on $C_2(u)$ and $C_3(u)$) to ensure that
\begin{align*}
\lVert |\nabla|^{3/4} u_{\leq N_1}\rVert_{L_t^2L_x^4}&\leq C_0\inf_{t\in I_0} \lVert (u_{\leq N_1}(t),\partial_t u_{\leq N_1}(t))\rVert_{\dot{H}_x^{3/2}\times\dot{H}_x^{1/2}}+\frac{2C(u)}{3}(N_1|I_0|)^{1/2}\\
&\hspace{0.4in}+\frac{2C(u)}{3}+\frac{C_4(u)}{c_0^{1/2}}(N_1|I_0|)^{1/2}. 
\end{align*}
We now choose $C(u)$ large enough so that
\begin{align*}
C(u)>\max\{\frac{3C_{4}(u)}{c_0(\eta_0)^{1/2}},3C_0\lVert (u,u_t)\rVert_{L_t^\infty(I;\dot{H}_x^{3/2}\times\dot{H}_x^{1/2})}\}
\end{align*}
With such a choice of $C(u)$ we obtain
\begin{align}
\lVert |\nabla|^{3/4} u_{\leq N_1}\rVert_{L_t^2L_x^4}&\leq C(u)(1+(N_1|I_0|)^{1/2}),
\end{align}
completing the induction.

We now turn to ($\ref{goal2}$).  Let $\eta>0$ be given and fix $N_0=N_0(\eta)>0$ to be determined later in the argument.  Let $N\leq N_0$ be given and recall that ($\ref{goal1}$) is satisfied for all $N>0$.  As a consequence, ($\ref{eqabcad1}$) is satisfied for any $\eta_0\in (0,\frac{1}{2}]$ with $N_1$ replaced by $N$.  More precisely, after setting
\begin{align*}
f(N)=\lVert (u_{\leq N},\partial_t u_{\leq N})\rVert_{L_t^\infty(\dot{H}_x^{3/2}\times\dot{H}_x^{1/2})}+\sup_{J_k\subset I} \lVert u_{\leq N}\rVert_{L_t^6(J_k;L_x^6)},
\end{align*}
we have
\begin{align}
&\nonumber \lVert |\nabla|^{3/4}u_{\leq N}\rVert_{L_t^2L_x^4}\\
\nonumber &\hspace{0.2in}\lesssim_u f(N)+\eta_0^{3/4}((N|I_0|)^{1/2}+\eta_0^{1/2})\\
\label{qwer}&\hspace{0.4in}+\eta_0^2(\eta_0^{-1/2}(N|I_0|)^{1/2}+1)+\frac{f(N)}{c_0^{1/2}}(N|I_0|)^{1/2}
\end{align}
for any $\eta_0\in (0,\frac{1}{2}]$.  We next show that $f(N)\rightarrow 0$ as $N\rightarrow 0$.  Indeed, invoking the Strichartz inequality ($\ref{str}$) and using the decomposition $u=u_{\leq N^{1/2}}+u_{>N^{1/2}}$ followed by the Bernstein inequalities,
\begin{align*}
f(N)&\lesssim \lVert (u_{\leq N},\partial_t u_{\leq N})\rVert_{L_t^\infty(\dot{H}_x^{3/2}\times\dot{H}_x^{1/2})}+\sup_{J_k\subset I} \lVert |\nabla|^{5/4}P_{\leq N}[u(t)^3]\rVert_{L_t^2(J_k;L_x^{4/3})}\\
&\lesssim \lVert (u_{\leq N},\partial_t u_{\leq N})\rVert_{L_t^\infty(\dot{H}_x^{3/2}\times\dot{H}_x^{1/2})}\\
&\hspace{0.2in}+\sup_{J_k\subset I} \bigg(\lVert |\nabla|^{5/4} P_{\leq N}[u_{>N^{1/2}}(t)^3]\rVert_{L_t^2(J_k;L_x^{4/3})}\\
&\hspace{0.4in}+\lVert |\nabla|^{5/4}P_{\leq N}[u_{>N^{1/2}}u_{\leq N^{1/2}}u]\rVert_{L_t^2(J_k;L_x^{4/3})}\\
&\hspace{0.4in}+\lVert |\nabla|^{5/4}P_{\leq N}[u_{\leq N^{1/2}}^3]\rVert_{L_t^2(J_k;L_x^{4/3})}\bigg)\\
&\lesssim \lVert (u_{\leq N},\partial_t u_{\leq N})\rVert_{L_t^\infty(\dot{H}_x^{3/2}\times\dot{H}_x^{1/2})}\\
&\hspace{0.2in}+\sup_{J_k\subset I} \bigg(N^{5/4}\lVert u\rVert_{L_t^6(J_k;L_x^6)}^2\lVert u_{>N^{1/2}}\rVert_{L_t^6(J_k;L_x^{12/5})}\\
&\hspace{0.4in}+\lVert u\rVert_{L_t^4(J_k;L_x^{20/3})}^2\lVert |\nabla|^{5/4}u_{<N^{1/2}}\rVert_{L_t^\infty L_x^{20/9}}\bigg)
\end{align*}
for any $N>0$.  We then bound the second term by using the Bernstein inequalities followed by $(\ref{eqIk})$ and the analogous bounds $\lVert |\nabla|^{5/4}u\rVert_{L_t^6(J_k;L_x^{12/5})}\lesssim_u 1$ and $\lVert u\rVert_{L_t^4(J_k;L_x^{20/3})}\lesssim_u 1$ to obtain
\begin{align*}
f(N)&\lesssim_u \lVert (u_{\leq N},\partial_t u_{\leq N})\rVert_{L_t^\infty(\dot{H}_x^{3/2}\times\dot{H}_x^{1/2})}+N^{5/8},
\end{align*}
which tends to $0$ as $N\rightarrow 0$ as a consequence of Remark \ref{rem_limit}.  With this limit in hand, we choose $\eta_0$ small enough to ensure $\eta_0^{3/4}<\eta$ and $N_0$ small enough to guarantee that $N<N_0$ implies $f(N)<\min \{\eta,\eta c_0(\eta_0)^{1/2}\}$.  The inequality ($\ref{qwer}$) then gives
\begin{align*}
\lVert |\nabla|^{3/4}u_{\leq N}\rVert_{L_t^2L_x^4}&\lesssim_u \eta(1+(N|I_0|)^{1/2})
\end{align*}
as desired.
\end{proof}

\begin{proof}[Proof of Corollary \ref{cor_floc_str}.]
We note that interpolation gives
\begin{align}
\lVert u_{> N}\rVert_{L_t^3L_x^{30/7}}&\lesssim \lVert u_{>N}\rVert_{L_t^\infty L_x^5}^{1/3}\lVert u_{> N}\rVert_{L_t^2L_x^4}^{2/3}\label{eqabcad2a}
\intertext{and}
\lVert u_{> N}\rVert_{L_t^4L_x^{20/7}}&\lesssim \lVert |\nabla|^{5/4}u_{> N}\rVert_{L_t^\infty L_x^{20/9}}^{1/2}\lVert |\nabla|^{-5/4}u_{> N}\rVert_{L_t^2L_x^4}^{1/2}.\label{eqabcad2}
\end{align}
The Sobolev inequality followed by the boundedness of the Littlewood-Paley projection then yields
\begin{align*}
\lVert |\nabla|^{5/4}u_{> N}\rVert_{L_t^\infty L_x^{20/9}}&\lesssim \lVert (u,u_t)\rVert_{L_t^\infty (\dot{H}_x^{3/2}\times\dot{H}_x^{1/2})}.
\end{align*}
On the other hand, the Bernstein inequalities along with Lemma $\ref{floc_str}$ give the bounds
\begin{align*}
\lVert u_{>N}\rVert_{L_t^2L_x^4}&\leq \sum_{M>N} M^{-3/4}\lVert |\nabla|^{3/4}u_M\rVert_{L_t^2L_x^4}\\
&\lesssim \sum_{M>N} M^{-3/4}\lVert |\nabla|^{3/4}u_{\leq 2M}\rVert_{L_t^2L_x^4}\\
&\lesssim_u \sum_{M>N} M^{-3/4}(1+(M|I_0|)^{1/2})\\
&\lesssim_u N^{-3/4}(1+N|I_0|)^{1/2}
\end{align*}
and (by an identical argument)
\begin{align*}
\lVert |\nabla|^{-5/4}u_{> N}\rVert_{L_t^2L_x^4}&\lesssim_u N^{-2}(1+N|I_0|)^{1/2}.
\end{align*}
Thus, we obtain
\begin{align*}
(\ref{eqabcad2a})&\lesssim_u N^{-1/2}(1+N|I_0|)^{1/3},&(\ref{eqabcad2})&\lesssim_u N^{-1}(1+N|I_0|)^{1/4}
\end{align*}
as desired.

The bounds on $\lVert \nabla u_{\leq N}\rVert_{L_{t,x}^3}$ and $\lVert \nabla u_{\leq N}\rVert_{L_t^4L_x^{20/7}}$ are obtained by interpolating ($\ref{goal2}$) with the a priori bound $(u,u_t)\in L_t^\infty(\dot{H}_x^{3/2}\times \dot{H}_x^{1/2})$.
\end{proof}

\section{Frequency-localized Morawetz estimate}
\label{sec_mor}

In this section, we obtain a frequency localized Morawetz estimate.  The proof of this result is inspired by the recent work of Visan \cite{Visan4} on the energy critical NLS.  

We begin by deriving a general form of the classical Morawetz estimate; for the classical form, see \cite{M,M2}.  To obtain this, when $u$ is a solution to $u_{tt}-\Delta u+\mathcal{N}=0$, we set
\begin{align*}
M(t)&=\int_{\mathbb{R}^5} -a_j(x)u_t(t,x)u_j(t,x)-\frac{1}{2}a_{jj}(x)u(t,x)u_t(t,x)dx,
\end{align*}
where $a:\mathbb{R}^5\rightarrow\mathbb{R}$, subscripts indicate partial derivatives, and we have used the summation convention.  A brief calculation then yields the identity
\begin{align*}
\frac{dM}{dt}(t)&=\int_{\mathbb{R}^5} a_{jk}(x)u_j(t,x)u_k(t,x)+\frac{1}{2}a_j(x)\{\mathcal{N},u\}_j-\frac{1}{4}a_{jjkk}(x)u(t,x)^2dx,
\end{align*}
with $\{f,g\}:=f\nabla g-g\nabla f$, where the subscript on $\{\mathcal{N},u\}$ denotes the $j$th component.  Taking $a(x)=|x|$, integrating in time, and using the fundamental theorem of Calculus, we then have
\begin{align}
\nonumber &\int_I\int_{\mathbb{R}^5} \left(\frac{\delta_{jk}}{|x|}-\frac{x_jx_k}{|x|^3}\right)u_j(t,x)u_k(t,x)+\frac{x_j\{\mathcal{N},u\}_j}{2|x|}+\frac{8}{|x|^3}u(t,x)^2dxdt\\
&\hspace{2.4in}\lesssim \sup_{t\in I} |M(t)|
\label{eq123a}
\end{align}
for every $I\subset \mathbb{R}$.  Moreover, the triangle inequality followed by the Cauchy-Schwartz and Hardy inequalities give
\begin{align}
\label{eq123b}|M(t)|&\lesssim \lVert u_t\rVert_{L_t^\infty L_x^2}\lVert \nabla u\rVert_{L_t^\infty L_x^2}
\end{align}
for all $t\in I$.  Combining ($\ref{eq123a}$) with ($\ref{eq123b}$), observing that the first term on the left hand side of ($\ref{eq123a}$) is non-negative and invoking an approximation argument, we obtain 
\begin{lemma}[Morawetz estimate]
Suppose $u:I\times\mathbb{R}^5\rightarrow\mathbb{R}$ solves $u_{tt}-\Delta u+\mathcal{N}=0$.  Then,
\begin{align}
\label{eqmor}\int_I\int_{\mathbb{R}^5} \frac{x\cdot \{\mathcal{N},u\}}{|x|}dxdt&\lesssim \lVert u_t\rVert_{L_t^\infty L_x^2}\lVert \nabla u\rVert_{L_t^\infty L_x^2}.
\end{align}
\end{lemma}

We also recall the following Hardy-type bound, which will be used to estimate the error terms resulting from the frequency localization.
\begin{proposition}[Hardy-type bound, \cite{Beckner}]  Fix $1<p<\infty$, and $0\leq \alpha<5$.  Then there exists $C=C(\alpha,p)>0$ such that for every $g\in\mathcal{S}(\mathbb{R}^5)$,
\begin{align}
 \lVert |x|^{-\alpha/p}g\rVert_{L_x^p(\mathbb{R}^5)}&\leq C(\alpha,p)\lVert |\nabla|^{\alpha/p}g\rVert_{L_x^p(\mathbb{R}^5)}.\label{hardy}
\end{align}
\end{proposition}

In particular, we prove the following:
\begin{lemma}[Frequency localized Morawetz estimate]
\label{lem1}
If $u:I\times\mathbb{R}^5\rightarrow\mathbb{R}$ is an almost periodic solution to (NLW) on $I^+=\cup J_k\subset \mathbb{R}$ with $N(t)=N_k\geq 1$ on each $J_k$ and $(u,u_t)\in L_t^\infty (\dot{H}_x^{3/2}\times\dot{H}_x^{1/2})$, then for any $\eta>0$ there exists $N_0=N_0(\eta)>0$ such that for all $N\leq N_0$ one has
\begin{align*}
\int_{I_0}\int_{\mathbb{R}^5} \frac{|u_{\geq N}(t)|^4}{|x|}dxdt&\leq \eta C(u)(N^{-1}+|I_0|)
\end{align*}
on any compact interval $I_0=\cup J_k$.
\end{lemma}

\begin{proof}
Fix a compact time interval $I_0=\cup J_k\subset I^+$.  In what follows, all spacetime norms will be taken over $I_0\times\mathbb{R}^5$, unless otherwise indicated.  Let $\eta>0$ be given, and fix $N_0>0$ to be determined later in the argument.  Let $N\leq N_0$ be given.  We begin by observing that the Morawetz estimate ($\ref{eqmor}$) applied to $u_{\geq N}$ yields
\begin{align}
\int_{I_0}\int_{\mathbb{R}^5} \frac{x\cdot \{P_{\geq N}[u(t)^3],u_{\geq N}\}}{|x|}dxdt&\lesssim \lVert \partial_tu_{\geq N}\rVert_{L_t^\infty L_x^2}\lVert \nabla u_{\geq N}\rVert_{L_t^\infty L_x^2}\label{eqa1}
\end{align}
Note that by Remark $\ref{rem_limit}$, we may choose $N_1>0$ so that $N\leq N_1$ implies 
\begin{align*}
\lVert (u_{\leq N},\partial_t u_{\leq N})\rVert_{L_t^\infty(I;\dot{H}_x^{3/2}\times\dot{H}_x^{1/2})}<\eta^{1/2}.
\end{align*}
Now, by choosing $N_0$ small enough so that $N_0<\eta N_1$, we may estimate the right hand side of ($\ref{eqa1}$) by 
\begin{align}
\nonumber &(\lVert \partial_t u_{N\leq \cdot <N_1}\rVert_{L_t^\infty L_x^2}+\lVert \partial_t u_{\geq N_1}\rVert_{L_t^\infty L_x^2})\cdot (\lVert \nabla u_{N\leq \cdot <N_1}\rVert_{L_t^\infty L_x^2}+\lVert \nabla u_{\geq N_1}\rVert_{L_t^\infty L_x^{2}})\\
\nonumber &\hspace{0.2in}\lesssim (N^{-1/2}\lVert \partial_t u_{<N_1}\rVert_{L_t^\infty \dot{H}_x^{1/2}}+N_1^{-1/2}\lVert \partial_t u_{\geq N_1}\rVert_{L_t^\infty \dot{H}_x^{1/2}})\\
\nonumber &\hspace{0.4in}\cdot (N^{-1/2}\lVert u_{<N_1}\rVert_{L_t^\infty \dot{H}_x^{3/2}}+N_1^{-1/2}\lVert u\rVert_{L_t^\infty \dot{H}_x^{3/2}})\\
&\hspace{0.2in}\lesssim_u \eta^2 N^{-1}.\label{bound1}
\end{align}

We now estimate the left hand side of ($\ref{eqa1}$).  For this, we use the identity
\begin{align*}
\{P_{\geq N}[u(t)^3],u_{\geq N}(t)\}&=\{u(t)^3,u(t)\}-\{u_{<N}(t)^3,u_{<N}(t)\}\\
&\hspace{0.2in}-\{u(t)^3-u_{<N}(t)^3,u_{<N}(t)\}\\
&\hspace{0.2in}-\{P_{< N}[u(t)^3],u_{\geq N}(t)\}
\end{align*}
to obtain
\begin{align}
&\nonumber \int_{I_0}\int_{\mathbb{R}^5} \frac{x\cdot \{P_{\geq N}[u(t)^3],u_{\geq N}(t)\}}{|x|}dxdt\\
&\nonumber \hspace{0.2in}=\int_{I_0}\int_{\mathbb{R}^5} \frac{x\cdot \{u(t)^3,u(t)\}}{|x|}-\frac{x\cdot \{u_{<N}(t)^3,u_{<N}(t)\}}{|x|}dxdt\\
&\nonumber \hspace{0.4in}-\int_{I_0}\int_{\mathbb{R}^5} \frac{x\cdot \{u(t)^3-u_{<N}(t)^3,u_{<N}(t)\}}{|x|}dxdt\\
\label{eqa2}&\hspace{0.4in}-\int_{I_0}\int_{\mathbb{R}^5} \frac{x\cdot \{P_{<N}[u(t)^3],u_{\geq N}(t)\}}{|x|}dxdt.
\end{align}
A simple calculation then shows $\{f^3,f\}=-\frac{1}{2}\nabla[f^4]$, so that integrating the first two terms in ($\ref{eqa2}$) by parts gives
\begin{align}
&\int_{I_0}\int_{\mathbb{R}^5} \frac{x\cdot \{P_{\geq N}[u(t)^3],u_{\geq N}(t)\}}{|x|}dxdt\nonumber \\
&\hspace{0.2in}=\int_{I_0}\int_{\mathbb{R}^5} \frac{2(|u(t)|^4-|u_{<N}(t)|^4)}{|x|}dxdt\nonumber \\
&\hspace{0.4in}-\int_{I_0}\int_{\mathbb{R}^5} \frac{x\cdot \{u(t)^3-u_{<N}(t)^3,u_{<N}(t)\}}{|x|}dxdt\nonumber \\
&\hspace{0.4in}-\int_{I_0}\int_{\mathbb{R}^5} \frac{x\cdot \{P_{<N}[u(t)^3],u_{\geq N}(t)\}}{|x|}dxdt.\label{eqabcad234}
\end{align}
On the other hand, applying the decomposition $u=u_{<N}+u_{\geq N}$ gives
\begin{align*}
\int_{I_0}\int_{\mathbb{R}^5} \frac{|u_{\geq N}(t)|^4}{|x|}dxdt&\lesssim \int_{I_0}\int_{\mathbb{R}^5} \frac{|u(t)|^4-|u_{<N}(t)|^4}{|x|}dxdt\\
&\hspace{0.4in} +\sum_{i=1}^3 \int_{I_0}\int_{\mathbb{R}^5}\frac{|u_{<N}(t)|^{4-i}|u_{\geq N}(t)|^{i}}{|x|}dxdt.
\end{align*}
In view of ($\ref{eqa1}$) and ($\ref{eqabcad234}$), we therefore obtain the bound
\begin{align*}
\int_{I_0}\int_{\mathbb{R}^5} \frac{|u_{\geq N}(t)|^4}{|x|}dxdt&\lesssim \eta N^{-1}+\sum_{i=1}^3 (I)_i+(II)+(III),
\end{align*}
where we have set
\begin{align*}
(I)_i&=\int_{I_0}\int_{\mathbb{R}^5}\frac{|u_{<N}(t)|^{4-i}|u_{\geq N}(t)|^{i}}{|x|}dxdt,\quad i=1,\cdots,3,\\
(II)&=\bigg|\int_{I_0}\int_{\mathbb{R}^5} \frac{x\cdot \{u(t)^3-u_{<N}(t)^3,u_{<N}(t)\}}{|x|}dxdt\bigg|,\quad \textrm{and}\\
(III)&=\bigg|\int_{I_0}\int_{\mathbb{R}^5} \frac{x\cdot \{P_{<N}[u(t)^3],u_{\geq N}(t)\}}{|x|}dxdt\bigg|.
\end{align*}

We estimate each of these terms individually.  For $(I)_i$, we use the H\"older inequality with the Hardy-type bound ($\ref{hardy}$), along with the Sobolev embedding and Corollary $\ref{cor_floc_str}$ (after choosing $N_0$ sufficiently small) to obtain $\ref{floc_str}$ to obtain the bounds
\begin{align*}
(I)_1&\lesssim \lVert u_{\geq N}\rVert_{L_t^4L_x^{20/7}}\left\lVert \frac{u_{< N}}{|x|^{1/3}}\right\rVert_{L_t^4 L_x^{60/13}}^3\\
&\lesssim \lVert u_{\geq N}\rVert_{L_t^4L_x^{20/7}}\lVert |\nabla|^{1/3}u_{< N}\rVert_{L_t^4L_x^\frac{60}{13}}^3\\
&\lesssim \lVert u_{\geq N}\rVert_{L_t^4L_x^{20/7}}\lVert \nabla u_{<N}\rVert_{L_t^4L_x^{20/7}}^3\\
&\lesssim_u \eta^3 (N^{-1}+|I_0|).
\end{align*}

For the $(I)_2$, we write
\begin{align*}
(I)_2&\lesssim \int_{I_0}\int_{\mathbb{R}^5} \frac{|u_{<N}(t)|\, |u_{\geq N}(t)|}{|x|}(|u_{<N}(t)|^2+|u_{\geq N}(t)|^2)dxdt\\
&\lesssim (I)_1+(I)_3,
\end{align*}
while for the term $(I)_3$, we note that for each $\epsilon>0$,
\begin{align*}
(I)_3&\lesssim \int_{I_0}\int_{\{x:|u_{<N}(t,x)|\leq \epsilon |u_{\geq N}(t,x)|\}}\frac{|u_{<N}(t)|\, |u_{\geq N}(t)|^3}{|x|}dxdt\\
&\hspace{0.2in}+\int_{I_0}\int_{\{x:|u_{<N}(t,x)|>\epsilon |u_{\geq N}(t,x)|\}}\frac{|u_{<N}(t)|\, |u_{\geq N}(t)|^3}{|x|}dxdt\\
&\leq \epsilon\int_{I_0}\int_{\mathbb{R}^5} \frac{|u_{\geq N}(t)|^4}{|x|}dxdt+\frac{1}{\epsilon}(I)_1.
\end{align*}

We now estimate term $(II)$.  Using the identity 
\begin{align*}
\{u^3-u_{<N}^3,u_{<N}\}=2(u^3-u_{<N}^3)\nabla u_{<N}-\nabla ((u^3-u_{<N}^3)u_{<N}),
\end{align*}
we apply the triangle inequality and integrate by parts in the second term of the resulting integral to obtain
\begin{align*}
(II)&\lesssim \int_{I_0}\int_{\mathbb{R}^5} |(u(t)^3-u_{<N}(t)^3)|\,|\nabla u_{<N}(t)|dxdt\\
&\hspace{0.2in}+\int_{I_0}\int_{\mathbb{R}^5} \frac{|(u(t)^3-u_{<N}(t)^3)|\,|u_{<N}(t)|}{|x|}dxdt\\
&\lesssim \sum_{i=1}^3\int_{I_0}\int_{\mathbb{R}^5} |u_{\geq N}(t)|^{i}|u_{< N}(t)|^{3-i}|\nabla u_{<N}(t)|dxdt+\sum_{i=1}^3 (I)_i
\end{align*}

We now use the H\"older inequality, Sobolev embedding, and Corollary $\ref{cor_floc_str}$ to estimate the first term,
\begin{align*}
\lVert u_{\geq N}u_{<N}^2\nabla u_{<N}\rVert_{L_{t,x}^1}&\leq \lVert u_{\geq N}\rVert_{L_t^4L_x^{20/7}}\lVert u_{<N}\rVert_{L_t^2L_x^{10}}\lVert u_{<N}\rVert_{L_t^\infty L_x^5}\lVert \nabla u_{<N}\rVert_{L_t^4L_x^{20/7}}\\
&\lesssim_u \lVert u_{\geq N}\rVert_{L_t^4L_x^{20/7}}\lVert |\nabla|^{3/4}u_{<N}\rVert_{L_t^2L_x^4}\lVert \nabla u_{<N}\rVert_{L_t^4L_x^{20/7}}\\
&\lesssim_u \eta N^{-1}(1+N|I_0|),
\end{align*}
the second term,
\begin{align*}
\lVert u_{\geq N}^2u_{<N}\nabla u_{<N}\rVert_{L_{t,x}^1}&\leq \lVert u_{\geq N}\rVert_{L_t^4L_x^{20/7}}\lVert u_{\geq N}\rVert_{L_t^\infty L_x^5}\lVert u_{<N}\rVert_{L_t^2L_x^{10}}\lVert \nabla u_{<N}\rVert_{L_t^4L_x^{20/7}}\\
&\lesssim_u \lVert u_{\geq N}\rVert_{L_t^4L_x^{20/7}}\lVert |\nabla|^{3/4}u_{<N}\rVert_{L_t^2L_x^4}\lVert \nabla u_{<N}\rVert_{L_t^4L_x^{20/7}}\\
&\lesssim_u \eta N^{-1}(1+N|I_0|),
\end{align*}
and the third term,
\begin{align*}
\lVert u_{\geq N}^3\nabla u_{<N}&\leq \lVert u_{\geq N}\rVert_{L_t^3L_x^{30/7}}^2\lVert u_{\geq N}\rVert_{L_t^\infty L_x^5}\lVert \nabla u_{<N}\rVert_{L_{t,x}^3}\\
&\lesssim_u \eta N^{-1}(1+N|I_0|).
\end{align*}

Combining these estimates then gives
\begin{align*}
(II)&\lesssim_u \eta N^{-1}(1+N|I_0|)+\sum_{i=1}^3 (I)_i.
\end{align*}

To continue, we estimate the remaining term, $(III)$.  In a similar manner as above, we use the identity
\begin{align*}
\{P_{<N}[u(t)^3],u_{\geq N}(t)\}=\nabla ( P_{<N}[u(t)^3]u_{\geq N}(t))-2u_{\geq N}(t)\nabla P_{<N}[u(t)^3]
\end{align*}
and integrate by parts in the first term of the resulting integral to obtain
\begin{align*}
(III)&\lesssim \int_{I_0}\int_{\mathbb{R}^5} \frac{|P_{<N}[u(t)^3]u_{\geq N}(t)|}{|x|}dxdt+\int_{I_0}\int_{\mathbb{R}^5} |u_{\geq N}(t)\nabla P_{<N}[u(t)^3]|dxdt\\
&\lesssim \sum_{i=0}^3 \left\lVert \frac{P_{<N}[u_{<N}(t)^iu_{\geq N}(t)^{3-i}]u_{\geq N}}{|x|}\right\rVert_{L_{t,x}^1}+\lVert u_{\geq N}\nabla P_{<N}[u_{<N}^iu_{\geq N}^{3-i}]\rVert_{L_{t,x}^1}
\end{align*}
We estimate the terms containing the gradient and remark that the other terms may then be bounded through the use of the Hardy-type inequality ($\ref{hardy}$).  In particular, we apply the H\"older, Bernstein, and Sobolev inequalities along with Corollary $\ref{cor_floc_str}$ to obtain, for the first term (using the bound from $(\ref{bound1})$),
\begin{align*}
 \lVert u_{\geq N}\nabla P_{<N}[u_{\geq N}^3]\rVert_{L_{t,x}^1}&\leq \lVert u_{\geq N}\rVert_{L_t^{\infty}L_x^{10/3}}\lVert \nabla P_{<N}[u_{\geq N}^3]\rVert_{L_t^1L_x^{10/7}}\\
&\lesssim N\lVert u_{\geq N}\rVert_{L_t^\infty \dot{H}_x^1}\lVert u_{\geq N}\rVert_{L_t^3L_x^{30/7}}^3\\
&\lesssim_u \eta N^{-1}(1+N|I_0|),
\end{align*}
for the second term,
\begin{align*}
\lVert u_{\geq N}\nabla P_{<N}[u_{<N}u_{\geq N}^2]\rVert_{L_{t,x}^1}&\leq \lVert u_{\geq N}\rVert_{L_t^\infty L_x^5}\lVert \nabla P_{<N}[u_{<N}u_{\geq N}^2]\rVert_{L_t^1L_x^{5/4}}\\
&\lesssim_u N\lVert u_{<N}\rVert_{L_t^2L_x^{10}}\lVert u_{\geq N}\rVert_{L_t^4L_x^{20/7}}^2\\
&\lesssim_u N\lVert |\nabla|^{3/4}u_{<N}\rVert_{L_t^2L_x^4}\lVert u_{\geq N}\rVert_{L_t^4L_x^{20/7}}^2\\
&\lesssim_u \eta N^{-1}(1+N|I_0|),
\end{align*}
for the third term,
\begin{align*}
\lVert u_{\geq N}\nabla P_{<N}[u_{<N}^2u_{\geq N}]\rVert_{L_{t,x}^1}&\leq \lVert u_{\geq N}\rVert_{L_t^4L_x^{20/7}}\lVert \nabla P_{<N}[u_{<N}^2u_{\geq N}]\rVert_{L_t^{4/3}L_x^{20/13}}\\
&\lesssim N\lVert u_{\geq N}\rVert_{L_t^4L_x^{20/7}}^2\lVert u_{<N}\rVert_{L_t^4L_x^{20/3}}^2\\
&\lesssim N\lVert u_{\geq N}\rVert_{L_t^4L_x^{20/7}}^2\lVert \nabla u_{<N}\rVert_{L_t^4L_x^{20/7}}^2\\
&\lesssim_u \eta N^{-1}(1+N|I_0|),
\end{align*}
and for the fourth term,
\begin{align*}
\lVert u_{\geq N}\nabla P_{<N}[u_{<N}^3]\rVert_{L_{t,x}^1}&\leq \lVert u_{\geq N}\rVert_{L_t^4L_x^{20/7}}\lVert \nabla P_{<N}[u_{<N}^3]\rVert_{L_t^{4/3}L_x^{20/13}}\\
&=\lVert u_{\geq N}\rVert_{L_t^4L_x^{20/7}}\lVert u_{<N}^2\nabla u_{<N}\rVert_{L_t^{4/3}L_x^{20/13}}\\
&\leq \lVert u_{\geq N}\rVert_{L_t^4L_x^{20/7}}\lVert  u_{<N}\rVert_{L_t^\infty L_x^5}\lVert u_{<N}\rVert_{L_t^2L_x^{10}}\lVert \nabla u_{<N}\rVert_{L_t^4L_x^{20/7}}\\
&\lesssim_u \eta N^{-1}(1+N|I_0|).
\end{align*}

Combining these estimates, we obtain
\begin{align*}
\int_{I_0}\int_{\mathbb{R}^5} \frac{|u_{\geq N}(t)|^4}{|x|}dxdt&\lesssim_u \eta (N^{-1}+|I_0|)+C_\epsilon(I)_1+\epsilon\int_{I_0}\int_{\mathbb{R}^5} \frac{|u_{\geq N}(t)|^4}{|x|}dxdt.
\end{align*}
Choosing $\epsilon$ sufficiently small, we obtain
\begin{align*}
\int_{I_0}\int_{\mathbb{R}^5} \frac{|u_{\geq N}(t)|^4}{|x|}dxdt&\lesssim_u 3\eta (N^{-1}+|I_0|)
\end{align*}
as desired.
\end{proof}

\section{Finite time blow-up solution}
\label{sec_fintime}
In this section, we rule out the existence of finite time blow-up solutions satisfying the properties stated in Theorem $\ref{thm_char_intervals}$.  Arguing as in \cite{BulutCubic,KM,KV3,KVradial}, this is accomplished by showing that such solutions must have zero energy, which in the defocusing case implies that the solution must be identically zero, contradicting its blow up.  In particular, we have the following theorem:
\begin{theorem}
Suppose that $u$ is an almost periodic solution to (NLW) with maximal interval of existence $I$, satisfying the properties given in Theorem $\ref{thm_char_intervals}$.  Then the case $|I^+|<\infty$ cannot occur.
\end{theorem}

\begin{proof}
Let $u$ be given as stated and suppose to the contrary that $|I^+|<\infty$.  By the time reversal and scaling symmetries we may assume that $\sup I=1$.  

We first show that
\begin{align}
\supp u(t,\cdot),\quad \supp u_t(t,\cdot)\subset \overline{B(0,1-t)},\quad 0<t<1.\label{suppgoal}
\end{align}
Indeed, the almost periodicity of $u$ in the form of Remark \ref{rem_apalt} gives that for all $\epsilon>0$ there exists $R=R(\epsilon)>0$ such that for every $0<s<1$ we have
\begin{align*}
\int_{|x|\geq \frac{R}{N(s)}} |\nabla u(s,x)|^{5/2}+|u_t(s,x)|^{5/2}dx< \epsilon.
\end{align*}
An invocation of the finite speed of propagation (see, for instance, \cite[Proposition $5.1$]{BulutCubic}) then gives
\begin{align}
\int_{|x|\geq \frac{R}{N(s)}+s-t} |\nabla u(t,x)|^{5/2}+|u_t(t,x)|^{5/2}dx\leq \epsilon\label{eqa}
\end{align}
whenever $0<t<s<1$, yielding
\begin{align*}
\limsup_{s\rightarrow 1} \int_{|x|\geq \frac{R}{N(s)}+s-t} |\nabla u(t,x)|^{5/2}+|u_t(t,x)|^{5/2}dx\leq \epsilon
\end{align*}
for $t\in (0,1)$.  On the other hand, recalling $N(t)\rightarrow\infty$ as $t\rightarrow 1$ (a consequence of the local theory and the almost periodicity), for all $t\in(0,1)$ and $\eta>0$ we have
\begin{align*}
\{x:|x|\geq 1-t+\eta\}\subset \{x:|x|\geq \frac{R}{N(s)}+s-t\}.
\end{align*}
when $s=s(t,\eta)$ is sufficiently close to $1$.  Combining this inclusion with ($\ref{eqa}$) and letting $\eta$ and $\epsilon$ tend to zero, we obtain
\begin{align*}
\int_{|x|\geq 1-t} |\nabla u(t)|^{5/2}+|u_t(t)|^{5/2}dx=0,
\end{align*}
which in turn yields that $(u(t,\cdot)$ is constant a.e. on $\{x:|x|\geq 1-t\}$ as well as $\supp u_t(t,\cdot)\subset \overline{B(0,1-t)}$.  To bound the support of $u$, we note that $u$ belongs to $L_x^\infty L_x^d$, which gives $(\ref{suppgoal})$.

To continue, by ($\ref{suppgoal}$), we write the energy by
\begin{align*}
E(u,u_t)&=\int_{|x|\leq 1-t} \frac{1}{2}|\nabla u(t)|^2+\frac{1}{2}|u_t(t)|^2+\frac{1}{4}|u(t)|^4dx\\
&\lesssim (1-t)[\lVert \nabla u(t)\rVert_{L_x^{5/2}(\mathbb{R}^5)}^2+\lVert u_t(t)\rVert_{L_x^{5/2}(\mathbb{R}^5)}^2+\lVert u(t)\rVert_{L_x^d(\mathbb{R}^5)}^4]\\
&\lesssim_u 1-t
\end{align*}
where we have used the a priori bound $(u,u_t)\in L_t^\infty(\dot{H}_x^{3/2}\times\dot{H}_x^{s_c-1})$.  Letting $t\rightarrow 1$ and using the conservation of energy, we obtain $u\equiv 0$, contradicting its blowup.
\end{proof}

\section{Infinite time blow-up solution}
\label{sec_inftime}
In this section, we consider the second class of solutions identified in Theorem $\ref{thm_char_intervals}$, almost periodic solutions to (NLW) which blow up in infinite time.  By making use of a frequency localized variant of the concentration of potential energy along with the frequency localized Morawetz estimate obtained in Section $\ref{sec_mor}$, we obtain a bound on the length of the maximal interval of existence, contradicting the assumption of infinite time blowup.  When combined with the results of the previous section, this completes the proof of Theorem $\ref{thm_main}$.  In particular, we prove the following theorem:
\begin{theorem}
There is no solution $u$ to (NLW) satisfying the properties of Theorem $\ref{thm_char_intervals}$ with $|I^+|=\infty$.
\end{theorem}

\begin{proof}
Suppose to the contrary that such a solution $u$ existed.  We begin by showing that there exists $C>0$ and $N_0>0$ such that for all $N\leq N_0$ and every $k\geq 1$,
\begin{align}
\int_{J_k}\int_{|x|\leq C/N_k} |u_{\geq N}(t,x)|^4dxdt\gtrsim_u N_k^{-2}.\label{eqa3a}
\end{align}
To show this claim, we recall that \cite[Lemma $2.6$]{KVradial} gives the existence of $C>0$ such that for every $k\geq 1$,
\begin{align*}
\int_{J_k}\int_{|x|\leq C/N_k} |u(t,x)|^4dxdt\gtrsim_u N_k^{-2}.
\end{align*}
An application of Minkowski's inequality then gives
\begin{align}
\nonumber &\left(\int_{J_k}\int_{|x|\leq C/N_k} |u_{\geq N}(t)|^4dxdt\right)^{1/4}\\
\nonumber &\hspace{0.2in}=\left(\int_{J_k}\int_{|x|\leq C/N_k} |u(t)-u_{\leq N/2}(t)|^4dxdt\right)^{1/4}\\
\nonumber &\hspace{0.2in}\geq \left(\int_{J_k}\int_{|x|\leq C/N_k} |u(t)|^4dxdt\right)^{1/4}-\left(\int_{J_k}\int_{|x|\leq C/N_k} |u_{\leq N/2}(t)|^4dxdt\right)^{1/4}\\
\label{eqa3}&\hspace{0.2in}\gtrsim_u N_k^{-1/2}-\left(\int_{J_k}\int_{|x|\leq C/N_k} |u_{\leq N/2}(t)|^4dxdt\right)^{1/4}.
\end{align}
On the other hand, fixing $\eta_1>0$ and applying H\"older's inequality along with Remark $\ref{rem_limit}$, we obtain that for $N$ sufficiently small
\begin{align*}
\int_{|x|\leq C/N_k} |u_{\leq N/2}(t)|^4dx&\lesssim_u N_k^{-1}\left(\int_{|x|\leq C/N_k} |u_{\leq N/2}(t)|^{5}dx\right)^{4/5}\\
&\lesssim_u N_k^{-1}\lVert u_{\leq N/2}\rVert_{L_t^\infty\dot{H}_x^{3/2}}^4\lesssim_u \frac{\eta_1^4}{N_k}.
\end{align*}
This implies the bound
\begin{align*}
\int_{J_k}\int_{|x|\leq C/N_k} |u_{\leq N/2}(t,x)|^4dxdt&\lesssim \frac{\delta\eta_1^4}{N_k^2},
\end{align*}
so that, after choosing $\eta_1$ sufficiently small and substituting this bound into $(\ref{eqa3})$, we obtain ($\ref{eqa3a}$).

We now fix $\eta>0$ to be determined later in the argument and recall that Lemma $\ref{lem1}$ implies the existence of $N_1\in (0,N_0)$ such that for all $N\leq N_1$ and $I_0=\cup J_k\subset I$ compact,
\begin{align}
\int_{I_0}\int_{\mathbb{R}^5} \frac{|u_{\geq N}(t)|^4}{|x|}dxdt&\lesssim_u\eta(N^{-1}+|I_0|) .\label{eq1a}
\end{align}
Combining $(\ref{eq1a}$) with $(\ref{eqa3a})$ then gives
\begin{align}
\nonumber \eta\bigg(N^{-1}+|I_0|\bigg)&\gtrsim_u \sum_{J_k\subset I_0} \int_{J_k}\int_{|x|\leq C/N_k} \frac{|u_{\geq N}(t)|^4}{|x|}dxdt\\
\nonumber &\gtrsim_u \sum_{J_k\subset I_0} N_k\int_{J_k}\int_{|x|\leq C/N_k} |u_{\geq N}(t)|^4dxdt\\
\nonumber &\gtrsim_u \sum_{J_k\subset I_0} N_k^{-1}\\
\label{eqasdf}&\gtrsim_u |I_0|-2\delta
\end{align}
for all $N\leq N_1$.  Choosing $\eta$ sufficiently small (depending on the constant in ($\ref{eqasdf}$)), we obtain the bound
\begin{align*}
|I_0|&\lesssim_u N^{-1}+1
\end{align*}
for all $N\leq N_1$ and all $I_0$.  Fixing $N$ and letting $I_0$ tend to $I^+$ then gives the desired contradiction.
\end{proof}

\end{document}